\documentclass[10pt]{amsart}
\usepackage{amsmath, amssymb, amsgen, amscd, xspace, color, epsfig}

\makeatletter
\newcommand{\rmnum}[1]{\romannumeral #1}
\newcommand{\Rmnum}[1]{\expandafter\@slowromancap\romannumeral #1@}

\newcommand{\pt}{\partial}

\newcommand{\ad}{\operatorname{ad}}

\newcommand{\id}{\operatorname{id}}

\newcommand{\str}{\mathrm{str}}

\newcommand{\Ker}{\mathrm{Ker}}

\newcommand{\Imm}{\mathrm{Im}}

\newcommand{\ra}{\rightarrow}

\newcommand{\pf}{\begin{proof}}
\newcommand{\epf}{\end{proof}}
\newcommand{\eq}{\begin{equation}}
\newcommand{\eeq}{\end{equation}}
\newcommand{\eqn}{\begin{equation*}}
\newcommand{\eeqn}{\end{equation*}}

\newcommand{\frb}{\mathfrak{b}}

\newcommand{\frg}{\mathfrak{g}}
\newcommand{\frh}{\mathfrak{h}}

\newcommand{\frn}{\mathfrak{n}}

\newcommand{\fru}{\mathfrak{u}}

\newcommand{\frgl}{\mathfrak{gl}}

\newcommand{\frsp}{\mathfrak{sp}}

\newcommand{\bbC}{\mathbb{C}}

\newcommand{\bbZ}{\mathbb{Z}}

\theoremstyle{plain}
\newtheorem{theorem}{Theorem}
\newtheorem{lemma}[theorem]{Lemma}
\newtheorem{cor}[theorem]{Corollary}
\newtheorem{prop}[theorem]{Proposition}
\theoremstyle{definition}
\newtheorem{definition}{Definition}

\theoremstyle{remark}

\begin{document}

\title[Dirac operators and cohomology for $\frgl(m|n)$]{Dirac operators and cohomology for the general linear superalgebra}

\author{Wei Xiao}
\thanks{This work is supported by NSFC Grant No. 11326059.}
\address{College of Mathematics and Computational Science, Shenzhen University,
Shenzhen, 518060, Guangdong, China}
\email{xiaow@szu.edu.cn}



\maketitle

\begin{abstract}
Vogan raised the idea of Dirac cohomology to study representations of semisimple Lie groups and Lie algebras. He conjectured that the infinitesimal character of Harish-Chandra modules are determined by their Dirac cohomology. Huang and Pand\v zi\'c proved this conjecture and initiated the research on Dirac cohomology for Lie superalgebras based on Kostant's results. The aim of the present paper is to study Dirac cohomology of unitary representations for the general linear superalgebra and its relation to nilpotent Lie superalgebra cohomology.
\end{abstract}

\section*{Introduction}

\renewcommand\thesubsection{\arabic{subsection}}
%
%
After Dirac discovered a matrix-valued first-order differential operator and had a remarkable success in the understanding of elementary particles, there have been various analogues of differential operators called Dirac operators. One striking example is the Dirac operator used to construct discrete series representations by Parthasarathy \cite{P} and Atiyah-Schmid \cite{AS}. Relative to Parthasarathy's geometric setting, Vogan introduced an algebraic version of the Dirac operator; and conjectured that the infinitesimal character of $(\frg,K)$-modules is determined by their Dirac cohomology \cite{V}. The conjecture was proved in \cite{HP1}. During the past twelve years, there have been many results of this nature. The Dirac cohomology turned out to be involved deeply with a few classical subjects of representation theory, like the discrete series and branching laws (see \cite{HP3, HPZ}). The relation between Dirac cohomology and nilpotent Lie algebra cohomology (Kostant's $\fru$-cohomology \cite{Ko1}) is also interesting. It was shown in \cite{HPR} that the Dirac cohomology of unitary modules is up to a twist isomorphic to $\fru$-cohomology for Hermitian types. Similar isomorphisms were obtained by Huang and Xiao in \cite{HX} for all the simple highest weight modules in the setting of cubic Dirac operator (see \cite{Ko2,Ko4}).

On the other hand, Kac's foundational papers \cite{Ka1, Ka2, Ka3} about Lie superalgebras and their representations had led to an enormous amount of work involving a growing list of researchers. A distinguished feature of the representation theory of Lie superalgebras is that Lie superalgebras have typical and atypical irreducible finite dimensional representations. The nilpotent Lie superalgebra cohomology groups play important roles in the determination of formal character of atypical representations (see \cite{Se1, Se2, B, SZ}). Dirac cohomology for Lie superalgebras was introduced by Huang and Pand\v zi\'c in \cite{HP2}. They defined Dirac cohomology for Lie superalgebras of Riemannian type (see \cite{Ko3}) and proved an analogue of Vogan's conjecture in the case of basic classical Lie superalgebras.

The aim of this paper is to study Dirac cohomology of unitary representations for the general linear superalgebra $\frgl(m|n)$ and its relation to nilpotent Lie superalgebra cohomology. More precisely, let $\frg=\frg_0\oplus\frg_1$ be a complex Lie superalgebra with the even part $\frg_0$ and the odd part $\frg_1$. For $\frgl(m|n)$, the odd part $\frg_1$ can be written as $\frg_1=\frg_+\oplus\frg_-$. Let $V$ be an irreducible unitary $(\frg,\frg_0)$-module. Then we have a Hodge decomposition and the Dirac cohomology of $V$ is up to a twist equal to $\frg_+$-cohomology and $\frg_-$-homology. Note that in this setting Cheng and Zhang found an explicit formula for the $\frg_+$-cohomology of unitarizable tensor representations. Therefore their calculation also gives the Dirac cohomology of the associated representations.

An outline of this paper is as follows. In Sect. 2 and 3, we recall the basic notions and properties of Lie superalgebras and corresponding Dirac cohomology. In Sect. 4, a correspondence between Weil representation and related polynomial algebra for $\frgl(m|n)$ is proved. A Hodge decomposition for $\frg_+$-cohomology and $\frg_-$-homology of unitary representations of $\frgl(m|n)$ is given in Sect. 5.

I wish to thank Prof. J.-S. Huang, for his encouragement and inspiring discussions.

%
%
\section{Lie superalgebras of Riemannian type}
%
%

In this section we outline the fundamental results on Lie superalgebras used in this paper, referring to \cite{Ko3} and \cite{HP3} for full
details. Let $\frg=\frg_0\oplus\frg_1$ be a complex Lie superalgebra with a bracket $[\cdot,\cdot]$. A bilinear form $B$ on $\frg$ is called supersymmetric if it is symmetric on $\frg_0$ and skew-symmetric on $\frg_1$, and $\frg_0$ and $\frg_1$ are orthogonal. The form $B$ is called invariant if $B([X,Y],Z)=B(X,[Y,Z])$ for all $X,Y,Z\in\frg$. We say $\frg$ is of Riemannian type if there exists a nondegenerate supersymmetric invariant bilinear form $B$ on $\frg$.

We call a subspace of $\frg_1$ a Lagrangian subspace if it is maximal isotropic. Fix a pair of complementary Lagrangian subspaces with bases $\partial_i, x_i$ for $i=1,\cdots,n$, such that
\[
B(\pt_i,x_j)=\frac{1}{2}\delta_{ij}.
\]
This notation is chosen so that the Weyl algebra $W(\frg_1)$ is generated by $\pt_i$ and $x_i$, with commutation relations:
\[
[x_i,x_j]_W=0;\quad [\pt_i,\pt_j]_W=0;\quad [\pt_i, x_j]_W=\delta_{ij}.
\]
The subscript $W$ is used to distinguish the commutators in $W(\frg_1)$ from the (totally different) bracket in $\frg$. With $\pt_i=\pt/\pt x_i$, we see that $W(\frg_1)$ can be identified with the algebra of differential operators with polynomial coefficients in variables $x_i$.

For the basis $\pt_1,\cdots,\pt_n,x_1,\cdots,x_n$ of $\frg_1$, the dual basis with respect to $B$ is $2x_1,\cdots,2x_n,-2\pt_1,\cdots,-2\pt_n$. The Casimir element of $\frg$ can then be defined as
\[
\Omega_\frg=\sum_kW_k^2+2\sum_i(x_i\pt_i-\pt_ix_i),
\]
where $W_k$ is an orthonormal basis of $\frg_0$ with respect to $B$. It is easy to check that $\Omega_\frg$ is contained in the center $Z(\frg)$ of the universal enveloping algebra $U(\frg)$ of $\frg$.

The adjoint action of $\frg_0$ on $\frg_1$ defines a map
\[
\nu:\frg_0\ra\frsp(\frg_1).
\]
We can embed $\frsp(\frg_1)$ into the Weyl algebra $W(\frg_1)$ as a Lie subalgebra consisting of quadratic elements. Start with the symmetrization map $\sigma:S(\frg_1)\ra W(\frg_1)$, where $S(\frg_1)$ is the symmetric algebra of $\frg_1$. It is a linear isomorphism obtained by first embedding $S(\frg_1)$ into the subset of symmetric tensors in the tensor algebra $T(\frg_1)$, and then projecting to $W(\frg_1)$. Next we show that $\sigma(S^2(\frg_1))$ is isomorphic to $\frsp(\frg_1)$. In fact, one can verify that
\begin{equation}\label{iso}
\sigma(x_ix_j)=x_ix_j,\quad \sigma(\pt_i\pt_j)=\pt_i\pt_j, \quad\sigma(\pt_ix_j)=\pt_i x_j-\frac{1}{2}\delta_{ij}=x_j\pt_i+\frac{1}{2}\delta_{ij}
\end{equation}
in the basis $(\pt_i,x_j)$. Considering the action of $\sigma(S^2(\frg_1))$ on $\frg_1\subset W(\frg_1)$ by commutators $[\cdot,\cdot]_W$ in $W(\frg_1)$, the associated matrices in the basis $(\pt_i,x_j)$ are
\[
\begin{aligned}
\sigma(x_ix_j)&\longleftrightarrow -E_{n+i\ j}-E_{n+j\ i};\\
\sigma(\pt_i\pt_j)&\longleftrightarrow~~ E_{i\ n+j}+E_{j\ n+i};\\
\sigma(\pt_ix_j)&\longleftrightarrow -E_{ij}-E_{n+j\ n+i},
\end{aligned}
\]
where $E_{kl}$ is the matrix with $1$ in the $k$-th row and $l$-th column and $0$ elsewhere. Therefore $\sigma(S^2(\frg_1))\simeq\frsp(\frg_1)$. In view of the map $\nu$ mentioned above, we obtain a Lie algebra morphism
\[
\alpha:\frg_0\ra W(\frg_1).
\]
Since $\alpha(\frg_0)\in\sigma(S^2(\frg_1))$, one can assume that
\begin{equation*}
\alpha(X)=\sum_{i,j}a_{ij}\sigma(x_ix_j)+\sum_{i,j}b_{ij}\sigma(\pt_i\pt_j)+\sum_{i,j}c_{ij}\sigma(\pt_ix_j),\quad X\in\frg_0.
\end{equation*}
To determine the coefficients, we apply $[\cdot,\pt_k]_W$ to both sides and get
\begin{equation*}
[\alpha(X),\pt_k]_W=[X,\pt_k]=-\sum_{i}a_{ik}x_i-\sum_{j}a_{kj}x_j-\sum_{i}c_{ik}\pt_i.
\end{equation*}
Then we apply $B(\cdot,\pt_l)$ and obtain $B([X,\pt_k],\pt_l)=1/2(a_{lk}+a_{kl})$, that is, $$a_{lk}+a_{kl}=2B(X,[\pt_k,\pt_l]).$$ Similarly, we get $$ b_{lk}+b_{kl}=2B(X,[x_k,x_l])$$ and $$c_{kl}=-2B(X,[x_k,\pt_l]).$$
With $(\ref{iso})$ in hand, the explicit formula for $\alpha$ is
\begin{equation}\label{phi}
\begin{aligned}
\alpha(X)=&\sum_{i,j}(B(X,[\pt_i,\pt_j])x_ix_j+B(X,[x_i,x_j])\pt_i\pt_j)\\
&-\sum_{i,j}2B(X,[x_i,\pt_j])x_j\pt_i-\sum_iB(X,[\pt_i,x_i]),\quad X\in\frg_0.
\end{aligned}
\end{equation}

Now we can define a diagonal embedding
\[
\frg_0\ra U(\frg)\otimes W(\frg_1)
\]
given by
\[
X\ra X\otimes1+1\otimes\alpha(X)=X_\Delta
\]
We denote by $\frg_{0\Delta}$ the image of $\frg_0$. Denote by $U(\frg_{0\Delta})$ the image of $U(\frg_0)$ and by $Z(\frg_{0\Delta})$ the image of the center $Z(\frg_0)$ of $U(\frg_0)$. Let $\Omega_{\frg_0}$ be the Casimir element for $\frg_0$. We denote by $\Omega_{\frg_{0\Delta}}$ the image of $\Omega_{\frg_0}$. Then
\[
\Omega_{\frg_{0\Delta}}=\sum_k(W_k^2\otimes1+2W_k\otimes\alpha(W_k)+1\otimes\alpha(W_k)^2).
\]
Kostant \cite{Ko3} proved that $C:=\sum_k\alpha(W_k)^2$ is a constant which equal to $1/8$ of the trace of $\Omega_0$ on $\frg_1$.

%
%
\section{Dirac cohomology for $(\frg,\frg_0)$}
%
%

In this section, we present the definition and fundamental results on Dirac cohomology for Lie superalgebras (\cite{HP2,HP3}).

The Dirac operator $D$ is defined to be the following element of $U(\frg)\otimes W(\frg_1)$:
\[
D=2\sum_{i=1}^n(\pt_i\otimes x_i-x_i\otimes\pt_i).
\]

Then $D$ is independent of the choice of basis of $\frg_1$ and is $\frg_0$-invariant.
The property of this Dirac operator is analogous to the case of reductive Lie algebras.

\begin{prop}[\cite{HP3}, Proposition 10.2.2]\label{D2}
Let $D\in U(\frg)\otimes W(\frg_1)$ be the Dirac operator. Then
\[
D^2=-\Omega_\frg\otimes1+\Omega_{\frg_{0\Delta}}-C,
\]
where $C$ is the constant mentioned above.
\end{prop}

Recall that the Weyl algebra $W(\frg_1)$ can be identified with the algebra of differential operators with polynomial coefficients in the $x_i$'s, where $i=1,\cdots,n$. Then we have a natural representation of $W(\frg_1)$ on the polynomial algebra $\bbC[x_1,\cdots,x_n]$. This is the Weil (or metaplectic) representation which we denote by $M(\frg_1)$. Note that $M(\frg_1)$ is $\bbZ_2$-graded. Let $M^+(\frg_1)$ and $M^-(\frg_1)$ be the submodules of $M(\frg_1)$ generated by homogeneous polynomials of even and odd degrees respectively.

\begin{definition}
Let $V$ be a representation of $\frg$. Consider the action of $D\in U(\frg)\otimes W(\frg_1)$ on $V\otimes M(\frg_1)$:
\[
D:V\otimes M(\frg_1)^\pm\ra V\otimes M(\frg_1)^\mp
\]
The Dirac cohomology of $V$ is the $\frg_0$-module
\[
H_D(V):=\Ker D/\Ker D\cap\Imm D
\]
\end{definition}
In particular, the $\bbZ_2$-grading of $M(\frg_1)$ implies a $\bbZ_2$-grading of $H_D(V)$, that is, $H_D(V)=H_D^+(V)\oplus H_D^-(V)$ with even part $H_D^+(V)$ and odd part $H_D^-(V)$.

\begin{theorem}[\cite{HP3}, Corollary 10.3.4 and Theorem 10.4.7]\label{Vconj}
Let $\frg$ be a basic classical Lie superalgebra with a Cartan subalgebra $\frh_0\subseteq\frg_0$. Let $W$ be the Weyl group of $(\frg,\frh_0)$. For any $z\in Z(\frg)$, there exists an algebra homomorphism $\zeta:
Z(\frg)\rightarrow Z(\frg_0)\cong Z(\frg_{0\Delta})$ and a $\frg_0$-invariant $a\in U(\frg)\otimes W(\frg_1)$, such that
\begin{equation*}
z\otimes 1-\zeta(z)=Da+aD.
\end{equation*}
Moreover, $\zeta$ fits into the following commutative diagram:
\begin{equation*}
\CD
  Z(\mathfrak{g}) @> \zeta >> Z(\mathfrak{g}_0) \\
  @V \gamma VV @V \gamma_{0} VV  \\
  S(\frh_0)^{W} @>\mathrm{id}>>
  S(\frh_0)^{W},
\endCD
\end{equation*}
where the vertical maps $\gamma$ and $\gamma_0$ are Harish-Chandra monomorphism and
isomorphism respectively.
\end{theorem}

For $\lambda\in\frh_0^*$, denote by $\chi_\lambda:Z(\frg)\ra\bbC$ the character $\chi_\lambda(z)=\lambda(\gamma(z))$ for $z\in Z(\frg)$. Similarly, denote by $\chi_\lambda^0:Z(\frg_0)\ra\bbC$ the character $\chi_\lambda^0(z_0)=\lambda(\gamma_0(z_0))$ for $z_0\in Z(\frg_0)$.

\begin{cor}[\cite{HP3}, Corollary 10.4.8]\label{Vconj2}
Let $V$ be a representation of a basic classical Lie superalgebra $\frg$, with infinitesimal character $\chi_\lambda$. Suppose that a $\frg_0$-module with $Z(\frg_0)$-infinitesimal character $\chi_\mu^{0}$ for $\mu\in\frh_0^*$ is contained in the Dirac cohomology $H_D(V)$. Then $\chi_\lambda(z)=\chi_\mu^0(\zeta(z))$.
\end{cor}

%
%
\section{The Weil representation}
%
%

From now on, assume that $\frg=\frgl(m|n)$ and $E_{ij}$ is the matrix in $\frgl(m|n)$ having $1$ in the $(i,j)$ position and $0$ elsewhere. Let $$\frg_0:=\sum_{i,j\leq m}\bbC E_{ij}+\sum_{k,l> m}\bbC E_{kl}$$ be the even part of $\frg$
with Cartan subalgebra $\frh_0:=\sum_{i}\bbC E_{ii}\subseteq\frg_0$. Denote
$$\frg_+:=\sum_{i\leq m<k}\bbC E_{ik}\quad \mbox{and}\quad \frg_-:=\sum_{i\leq m<k}\bbC E_{k i}.$$ Both of them are $\frg_0$-invariant super commutative subalgebras of $\frg$. The odd space $\frg_1=\frg_+\oplus\frg_-$. We also have the standard Borel subalgebra $$\frb:=\sum_{i\leq j}\bbC E_{ij}.$$ Denote $$\frn^+:=\sum_{i< j}\bbC E_{ij}\quad \mbox{and}\quad \frn^-:=\sum_{i>j}\bbC E_{ij}.$$ Then $\frn^+$ and $\frn^-$ are invariant under the adjoint action of $\frh_0$. We have
\begin{equation*}
\frg=\frn^+\oplus\frh_0\oplus\frn^-,\quad\mbox{and}\quad \frb=\frh_0\oplus\frn^+.
\end{equation*}
Thus the Harish-Chandra homomorphism $\gamma:Z(\frg)\ra S(\frh_0)$ can be defined in the standard way.

Let $\Delta\subset\frh_0^*$ be the root system of $(\frg, \frh_0)$, with positive root system $\Delta^+$ corresponding to $\frb$. The set $\Delta^+$ decomposes as $\Delta_0^+\cup\Delta_1^+$, where $\Delta_0^+$ and $\Delta_1^+$ denote the sets of the even and odd positive roots respectively. Set
\[
\rho_0=\frac{1}{2}\sum_{\beta\in\Delta_0^+}\beta,\quad
\rho_1=\frac{1}{2}\sum_{\beta\in\Delta_1^+}\beta\quad\mbox{and}\quad\rho=\rho_0-\rho_1.
\]

The supertrace function on $\frg$ can be defined as
\[
\str A=\sum_{i\leq m}a_{ii}-\sum_{j>m}a_{jj},
\]
where $a_{kl}$ is the $(k,l)$ entry of $A\in\frg$.
Therefore $$B(X, Y):=\frac{1}{2}\str(XY)$$ is a nondegenerate supersymmetric invariant bilinear form on $\frg$ (see Proposition 1.1.2 in \cite{Ka2}). Therefore $\frg$ is of Riemannian type. In particular, we can assume that
\begin{equation}\label{changeindex}
\pt_{(i-1)n+(k-m)}=E_{ik}\quad\mbox{and}\quad x_{(i-1)n+(k-m)}=E_{k i}
\end{equation}
for all $i\leq m<k$. They form basis of $\frg_+$ and $\frg_-$ respectively. Let $\bbC[x_1,\cdots,x_{mn}]$ be the polynomial algebra generated by $x_1,\cdots,x_{mn}$. Then it is a graded $\frg_0$-module under the natural adjoint action.

\begin{theorem}
Let $V$ be a representation of $\frg=\frgl(m|n)$, with $Z(\frg)$-infinitesimal character $\chi_\lambda$ for $\lambda\in\frh_0^*$. Suppose that $H_D(V)$ contains a $\frg_0$-submodule $N$ with $Z(\frg_0)$-infinitesimal character $\chi_\mu^0$ for $\mu\in\frh_0^*$. Then $\chi_\lambda=\chi_\mu$.
\end{theorem}

\begin{proof}
An argument similar to the one used in \cite{HP3} (10.3 and 10.4) shows that Theorem \ref{Vconj} and Corollary \ref{Vconj2} also hold for $\frgl(m|n)$. Let $v\in N$ represent a nonzero Dirac cohomology class. It follows from Theorem \ref{Vconj} that for all $z\in Z(\frg)$, there exists $a\in U(\frg)\otimes W(\frg_1)$ such that
\[
z\otimes1-\chi_\lambda(z)=(\zeta(z)-\chi_\mu^{0}(\zeta(z)))+(Da+aD)
+(\chi_\mu^{0}(\zeta(z))-\chi_\lambda(z)).
\]
Applying this identity to $v$, the left side becomes zero since $$(z\otimes1-\chi_\lambda(z))\cdot (V\otimes M(\frg_1))=0.$$
The right side is $(\chi_\mu^{0}(\zeta(z))-\chi_\lambda(z))v\ (\mathrm{mod}\ \Imm\, D)$. Then $(\chi_\mu^{0}(\zeta(z))-\chi_\lambda(z))v \in\Imm\ D\ (\subseteq\Ker\ D)$, which is zero in $H_D(V)$. Thus $\chi_\mu^{0}(\zeta(z))-\chi_\lambda(z)=0$. On the other hand, one has
\[
\chi_\mu^{0}(\zeta(z))=(\mu\circ\gamma_0\circ\zeta)(z)=(\mu\circ\id\circ\gamma)(z)=\chi_\mu(z).
\]
Then we conclude that $\chi_\lambda=\chi_\mu$.
\end{proof}

Recall the formula $(\ref{phi})$ of Lie algebra morphism
\[
\alpha:\frg_0\ra W(\frg_1).
\]
Put
$$\alpha_1(X)=-\sum_{i,j}2B(X,[x_i,\pt_j])x_j\pt_i$$ and
$$\alpha_2(X)=-\sum_iB(X,[\pt_i,x_i]).$$
for $X\in\frg_0$. Since $[\pt_i,\pt_j]=[x_i,x_j]=0$ in this case, we have $$\alpha(X)=\alpha_1(X)+\alpha_2(X).$$

\begin{lemma}\label{vf1}
Let $f$ be a polynomial in $\bbC[x_1,\cdots,x_{mn}]$. Then for any $X\in\frg_0$,
\[
\alpha_1(X)f=[X, f],
\]
where the left side is given by the action of Weil representation and the right side is given by the usual adjoint action.
\end{lemma}

\begin{proof}
By linearity, it suffices to consider the case when $f$ is a monomial, that is, $$\quad f=\prod_kx_k^{q_k},\qquad q_k\in\bbZ^{\geq0}~\mbox{for}~k=1,\cdots,mn.$$
Since $\frg_-$ is $\frg_0$-invariant, we can obtain $$[X, x_i]=-\sum_{j}2B([X,x_i],\pt_j)x_j=-\sum_{j}2B(X,[x_i,\pt_j])x_j.$$
Therefore, $\alpha_1(X)=\sum_i[X,x_i]\pt_i$.
It follows that
\[
\begin{aligned}
\alpha_1(X)f&=\sum_i[X,x_i]\pt_i\prod_kx_k^{q_k}\\
&=\sum_i[X,x_i]\frac{q_i}{x_i}\prod_kx_k^{q_k}\\
&=\sum_ix_1^{q_1}\cdots[X, x_i^{q_i}]\cdots x_{mn}^{q_{mn}}\\
&=[X, \prod_kx_k^{q_k}]=[X, f].
\end{aligned}
\]
\end{proof}

\begin{lemma}\label{vf2}
Let $\bbC_{-\rho_1}$ be the one-dimensional $\frg_0$-module with weight $-\rho_1$. Given $v\in\bbC_{-\rho_1}$, then
\[
\alpha_2(X)v=X\cdot v.
\]
\end{lemma}

\begin{proof}
First, we observe that $$\alpha_2(X)=-\sum_iB(X,[\pt_i,x_i])=-\sum_{i\leq m<j}B(X,E_{ii}+E_{jj}).$$
If $E_{kl}\in\frg_0$ and $k\neq l$, then $B(E_{kl},E_{ii}+E_{jj})=0$ for all $i\leq m<j$. So $\alpha_2(E_{kl})v=0=E_{kl}\cdot v.$ If $k=l$, then $E_{kk}\in\frh_0$. Let $\beta_i$ be the positive root associated with $\pt_i$. Then $\beta_i\in\Delta_1^+$ and
\[
\begin{aligned}
\alpha_2(E_{kk})v=&-\sum_iB(E_{kk},[\pt_i,x_i])v=-\sum_iB([E_{kk},\pt_i],x_i)v\\
=&-\sum_i\beta_i(E_{kk})B(\pt_i,x_i)v=-\frac{1}{2}\sum_i\beta_i(E_{kk})v\\
=&-\rho_1(E_{kk})v=E_{kk}\cdot v.
\end{aligned}
\]
\end{proof}

It follows immediately from Lemma \ref{vf1} and Lemma \ref{vf2} that the action of $\alpha(\frg_0)$ on $M(\frg_1)$ and the adjoint action $\ad\frg_0$ on $\bbC[x_1,\cdots,x_{mn}]$ differ by a twist of the one-dimensional character $\bbC_{-\rho_1}$. We have

\begin{prop}\label{g0 iso}
There exists a $\frg_0$-module isomorphism
\[
M(\frg_1)\simeq \bbC[x_1,\cdots,x_{mn}]\otimes\bbC_{-\rho_1}.
\]
\end{prop}

We see that the symmetric algebras $S(\frg_\pm)$ of $\frg_\pm$ are graded with $S(\frg_\pm)=\sum_{i=0}^{\infty}S^i(\frg_\pm)$, where $S^i(\frg_\pm)$ are homogeneous of degree $i$ in $\frg_\pm$. We can identify $\frg_+^*$ with $\frg_-$ by the pairing $2B(\cdot,\cdot):\frg_+\times\frg_-\ra\bbC\mathbb{}$. The identification is $\frg_0$-invariant since $B$ is invariant. Upon identifying $S(\frg_+^*)$ with $S(\frg_-)$, and then the polynomial algebra $\bbC[x_1,\cdots,x_{mn}]$, the cohomology group $H^i(\frg_+, V)$ is given by the complex $C=(\{V\otimes S^i(\frg_-)\}, d)$, where $d=\sum_{i}\pt_i\otimes x_i$ is $\frg_0$-invariant (cf. \cite{CZ}, 3.7). On the other hand, the homology group $H_i(\frg_-, V)$ is given by the complex $(\{V\otimes S^i(\frg_-)\}, \delta)$, with $\frg_0$-invariant differential operator $\delta=\sum_{i} x_i\otimes\pt_i$. Then the following lemma is an immediate consequence of Proposition \ref{g0 iso}.

\begin{prop}\label{hom iso}
If we consider $d$ and $\delta$ as operators on $V\otimes M(\frg_1)$, then as $\frg_0$-modules, the cohomology of $d$ is identified with $H^*(\frg_+, V)\otimes\bbC_{-\rho_1}$, while the homology of $\delta$ is identified with $H_*(\frg_-, V)\otimes\bbC_{-\rho_1}$.
\end{prop}

\section{Hodge decomposition for $\frg_+$-cohomology}

Now we consider Hermitian forms on $V\otimes M(\frg_1)$ for unitarizable module $V$. Let $\omega$ be the antilinear anti-involution of $U(\frg)$ defined by
\[
E_{ij}\ra E_{ji},\quad 1\leq i,j\leq m+n.
\]
In general, for any $a,b\in U(\frg)$, we have $\omega(ab)=\omega(b)\omega(a)$. Recall that we say a $\bbZ_2$-graded $\frg$-module $V$ is unitary if it admits a positive definite contravariant Hermitian form $\langle\cdot,\cdot\rangle_V$, where contravariance means that $\langle av,v'\rangle_V=\langle v,\omega(a)v'\rangle_V$ for all $a\in\frg$ and $v,v'\in V$. On the other hand, there is a unique positive definite contravariant Hermitian form $\langle\cdot,\cdot\rangle_M$ on $M(\frg_1)$, with $\langle1,1\rangle_{M}=1$, where contravariance means that $\langle af,f'\rangle_M=\langle f,\omega(a)f'\rangle_M$ for all $a\in\frg$ and $f,f'\in M$. Here, we should emphasize that the form is given explicitly by
\begin{equation*}
\begin{aligned}
\langle\prod_kx_k^{p_k},\prod_kx_k^{q_k}\rangle_M&=\prod_k p_k!\qquad\mbox{if}~ p_k=q_k~\mbox{for all}~k\\
&=0\quad\quad\quad\quad\quad~\mbox{otherwise},
\end{aligned}
\end{equation*}
where $p_k,q_k\in\bbZ^{\geq0}$.

Let $V$ be a $(\frg,\frg_0)$-module, that is, viewed as a $\frg_0$-module, $V$ is a direct sum of finite dimensional simple modules with finite multiplicities. We consider the tensor product Hermitian forms on $V\otimes M(\frg_1)$; this form will be denoted by $\langle\cdot,\cdot\rangle$.

\begin{lemma}
Let $V$ be a unitary $(\frg,\frg_0)$-module. With respect to the form $\langle\cdot,\cdot\rangle$ on $V\otimes M(\frg_1)$, the operators $d$ and $\delta$ are adjoints of each other. Hence the Dirac operator $D=2(d-\delta)$ is anti self-adjoint.
\end{lemma}
\begin{proof}
In view of $(\ref{changeindex})$, we have $\omega(\pt_i)=x_i$ and $\omega(x_i)=\pt_i$. Then the lemma follows from the fact that $\langle\cdot,\cdot\rangle_V$ and $\langle\cdot,\cdot\rangle_M$ are contravariant.
\end{proof}

Recall that $D^2=-\Omega_\frg\otimes1+\Omega_{\frg_{0\Delta}}-C$. Suppose that $\Omega_\frg$ acts on $V$ by a constant. Since $\Omega_{\frg_0\Delta}$ acts by a scalar on each irreducible $\frg_0$-submodules in $V\otimes M(\frg_1)$, the same is true for $D^2$.

\begin{lemma}\label{last prop}
If the $(\frg,\frg_0)$-module $V$ has infinitesimal character, then $V\otimes M(\frg_1)=\Ker\, D^2\oplus\Imm\,D^2$.
\end{lemma}

\begin{proof}
Both $V$ and $M(\frg_1)$ are direct sums of finite dimensional irreducible $\frg_0$-modules, so is the tensor product $V\otimes M(\frg_1)$. Then $V\otimes M(\frg_1)$ is a direct sum of eigenspaces for $D^2=-\Omega_\frg\otimes1+\Omega_{\frg_{0\Delta}}-C$. The zero eigenspace is $\Ker\, D^2$ and the sum of all the nonzero eigenspaces is $\Imm\, D^2$.
\end{proof}

\begin{lemma}\label{last lem}
For Dirac operator $D$, we have
\begin{itemize}
\item[$\mathrm{(\rmnum{1})}$] $\Ker\, D^2=\Ker\, D=\Ker\, d\cap\Ker\,\delta$;

\item[$\mathrm{(\rmnum{2})}$] With respect to the form $\langle\cdot,\cdot\rangle$, $\Imm\, d$ is orthogonal to $\Ker\,\delta$ and $\Imm\,\delta$, $\Imm\,\delta$ is orthogonal to $\Ker\, d$.
\end{itemize}
\end{lemma}

\begin{proof}
(\rmnum{1}) For any $a\in V\otimes M(\frg_1)$, it follows from $\langle Da,Da\rangle=\langle -D^2a, a\rangle$ that $Da=0$ if and only if $D^2a=0$. On the other hand, $\Ker\, d\cap\Ker\,\delta\subseteq\Ker\, D$ since $D=2(d-\delta)$. Conversely, if $Da=0$, then $da=\delta a$ and $\delta da=\delta^2a=0$. So $\langle da,da\rangle=\langle a,\delta da\rangle=0$. Hence $da=0$. Similarly we get $\delta a=0$.
(\rmnum{2}) It is an easy consequence of the fact that $d$ and $\delta$ are adjoints of each other.
\end{proof}

\begin{theorem}
Let $V$ be an irreducible unitary $(\frg, \frg_0)$-module. Then
\begin{itemize}
\item[$\mathrm{(\rmnum{1})}$] $V\otimes M(\frg_1)=\Ker\, D\oplus\Imm\, d\oplus\Imm\,\delta$;

\item[$\mathrm{(\rmnum{2})}$] $\Ker\, d=\Ker\, D\oplus\Imm\,d$;

\item[$\mathrm{(\rmnum{3})}$] $\Ker\, \delta=\Ker\, D\oplus\Imm\,\delta$.
\end{itemize}
In particular, there exists $\frg_0$-module isomorphisms:
\begin{equation*}
H_D(V)\simeq H^*(\frg_+,V)\otimes\bbC_{-\rho_1}\simeq H_*(\frg_-,V)\otimes\bbC_{-\rho_1}
\end{equation*}
\end{theorem}

\begin{proof}
(\rmnum{1}) In view of Lemma \ref{last lem}, we see that $\Ker\, D$, $\Imm\, d$ and $\Imm\,\delta$ are disjoint subspaces of $V\otimes M(\frg_1)$. Since $D=2(d-\delta)$, one has $\Imm\, D^2\subseteq\Imm\, D\subseteq\Imm\,d\oplus\Imm\,\delta$. It follows from Lemma \ref{last prop} and Lemma \ref{last lem} that
\begin{equation*}
V\otimes M(\frg_1)=\Ker\, D^2\oplus\Imm\, D^2\subseteq\Ker\, D\oplus\Imm\, d\oplus\Imm\,\delta.
\end{equation*}
Then (\rmnum{1}) follows and $\Imm\, D^2=\Imm\, D=\Imm\,d\oplus\Imm\,\delta$.
The formula (\rmnum{2}) is an obvious consequences of (\rmnum{1}) and Lemma \ref{last lem}, so is (\rmnum{3}). Thus $$H_D(V)=\Ker\, D\simeq\Ker\, d/\Imm\, d\simeq\Ker\, \delta/\Imm\, \delta.$$
The theorem is now evident from Proposition \ref{hom iso}.
\end{proof}



\end{document}